\documentclass[pdflatex,sn-mathphys-num]{sn-jnl}

\usepackage{graphicx}%
\usepackage{multirow}%
\usepackage{amsmath,amssymb,amsfonts}%
\usepackage{amsthm}%
\usepackage{mathrsfs}%
\usepackage[title]{appendix}%
\usepackage{xcolor}%
\usepackage{textcomp}%
\usepackage{manyfoot}%
\usepackage{booktabs}%
\usepackage{algorithm}%
\usepackage{algorithmicx}%
\usepackage{algpseudocode}%
\usepackage{listings}%

\usepackage{mathrsfs} 
\usepackage{enumitem,lineno}
\usepackage{graphicx}

\usepackage[T1]{fontenc}
\usepackage{amsfonts,amsmath,amssymb}
\usepackage{tikz,comment}
\usetikzlibrary{positioning, calc, trees, decorations.markings, patterns, arrows, arrows.meta}
\usepackage{diagmac2}

\newtheorem{theorem}{Theorem}
\newtheorem{proposition}{Proposition}
\newtheorem{lemma}{Lemma}%
\newtheorem{corollary}{Corollary}%
\newtheorem{definition}{Definition}%

\theoremstyle{thmstyletwo}%
\newtheorem{example}{Example}%
\newtheorem{remark}{Remark}%

\begin{document}

\title[Labeled Trees Generating Separable and Locally Finite Ultrametrics]{Labeled Trees Generating Separable and Locally Finite Ultrametrics}


\author*[1,2]{\fnm{Oleksiy} \sur{Dovgoshey}}\email{oleksiy.dovgoshey@gmail.com,oleksiy.dovgoshey@utu.fi}

\author[3]{\fnm{Olga} \sur{Rovenska}}\email{rovenskaya.olga.math@gmail.com}

\affil*[1]{\orgdiv{Department of Theory of Functions}, \orgname{Institute of Applied Mathematics and Mechanics of NASU}, \orgaddress{\city{Slovyansk}, \country{Ukraine}}}

\affil[2]{\orgdiv{Department of Mathematics and Statistics}, \orgname{University of Turku}, \orgaddress{\city{Turku}, \country{Finland}}}

\affil[3]{\orgdiv{Department of Mathematics and Modeling}, \orgname{Donbas State Engineering Academy}, \orgaddress{\city{Kramatorsk},  \country{Ukraine}}}


\abstract{We analyze the interplay between labeled trees and the ultrametric spaces they present. We provide characterizations of labeled trees that generate separable ultrametric spaces and those that generate locally finite ultrametric spaces. In particular, we establish an analog of K\"onig’s Infinity Lemma for locally finite ultrametric spaces generated by labeled trees.}




\maketitle

\section{Introduction}

According to \cite{GV}, any finite ultrametric space is isometrically describable by a Gurvich--Vyalyi representing tree.  The
Gurvich--Vyalyi representing trees form a subclass of finite trees equipped with a specific labeling of the vertex set.  Trees with labeled vertices have been extensively studied, resulting in numerous contributions (see, e.g., the survey in \cite{Gallian2024}). The corresponding geometric interpretation of the Gurvich--Vyalyi representation \cite{PD2014JMS} provides a framework for addressing various extremal problems related to finite ultrametric spaces \cite{DP2020pNUAA,DPT2016,DPT2017FPTA}.
Analogues of the Gurvich--Vyalyi representation and its geometric interpretation have recently been extended to the class of totally bounded ultrametric spaces \cite{Dovg2025}.

Infinite trees with positive real-valued edge labelings are commonly referred to as \(R\)-trees (see~\cite{Ber2019SMJ} for relevant results concerning \(R\)-trees and ultrametrics). The structural relationships between finite subtrees of \(R\)-trees and finite monotone rooted trees are described in~\cite{Dov2020TaAoG}. The categorical equivalence between trees and ultrametric spaces has been studied in \cite{H04,Lem2003AU}.

Ultrametric spaces generated by arbitrary nonnegative vertex labelings on finite and infinite trees were introduced in \cite{Dov2020TaAoG} and subsequently studied in \cite{DK2024DLPSSAG,DK2022LTGCCaDUS}.  
A characterization of totally bounded ultrametric spaces generated by labeled almost rays was provided in \cite{DV2025JMS}. Furthermore, \cite{DR2025USGbLSG} contains a metric characterization of ultrametric spaces generated by labeled star graphs.

Our work continues the aforementioned studies on the characterization of labeled trees generating ultrametric spaces, with a focus on the topological properties of these spaces.
The aim is to describe the class of trees that admit labelings which generate separable ultrametric spaces. Additionally, we establish a characterization of such trees in terms of the local finiteness of the corresponding ultrametrics.

The structure of the paper is as follows. Section~\ref{sec2} introduces and reminds the necessary definitions and results. 

Theorem~\ref{ggewj77} in Section~\ref{sec3} completely describes the structure of free trees which admit labelings generating separable ultrametric spaces. The labeled star graphs generating separable ultrametric are characterized in Corollary~\ref{dfuj247-1} of Section~\ref{sec3}.

Section~\ref{sec4} is devoted to the characterization of trees that generate locally finite ultrametric spaces. The necessary and sufficient conditions under which labeled rays and labeled star graphs generate locally finite ultrametric spaces are given in Proposition~\ref{52dcij}
and, respectively, in Proposition~\ref{dfuj247} of Section~\ref{sec4}.
Theorem~\ref{dster} provides an analog of K\"onig’s Infinity Lemma for the locally finite case. The main result of this section, Theorem~\ref{trees11}, describes the trees that admit labelings 
generating locally finite ultrametric spaces.

\section{Preliminaries}\label{sec2}

Throughout this text, we denote by \(\mathbb{R}^+\) the half-open interval \([0, \infty)\), by \(\mathbb{N}\) the set of all positive integers, and by $\aleph_0$ the cardinality of the set $\mathbb N$. In addition, we will write $(s_n)_{n \in \mathbb{N}}\subseteq A$ if $(s_n)_{n \in \mathbb{N}}$
is a sequence whose points are elements of the set $A$.

A \textit{metric} on a set $X$ is a function $d\colon X\times X\rightarrow \mathbb R^+$ such that for all \(x\), \(y\), \(z \in X\)
\begin{enumerate}[label=\textit{(\roman*)}, left=0pt]
\item $d(x,y)=d(y,x)$,
\item $(d(x,y)=0)\iff (x=y)$,
\item \(d(x,y)\leq d(x,z) + d(z,y)\).
\end{enumerate}

A metric space \((X, d)\) is called an \emph{ultrametric space} if the \emph{strong triangle inequality}
\[
d(x,y)\leq \max \{d(x,z),d(z,y)\}
\]
holds for all \(x\), \(y\), \(z \in X\). In this case, the function \(d\) is called \emph{an ultrametric} on \(X\).

Let \((X, d)\) be a metric space. A \emph{closed ball} with a \emph{radius} \(r > 0\) and a \emph{center} \(c \in X\) is the set
\[
B_r(c) = \{x \in X \colon d(c, x) \leq r\}.
\]


A set $A \subseteq X$ is said to be {\it bounded} if there exists $B_r(c) \subseteq X$ such that $A \subseteq B_r(c).$

We say that a set $S \subseteq X$ is a \emph{discrete subset} of $(X, d)$ if for every $s \in S$ there is $r > 0$ such that
\begin{equation}
    B_r(s) \cap S = \{s\}.
\end{equation}

A sequence \((x_n)_{n \in \mathbb{N}}\) of points in a metric space \((X, d)\) is said to be {\it convergent} to a point \(a \in X\),
\begin{equation*}
\lim_{n \to \infty} x_n = a,
\end{equation*}
if and only if 
\begin{equation*}
\lim_{n \to \infty} d(x_n, a) = 0.
\end{equation*}

Let \((X, d)\) be a metric space, and let \(S \subseteq X\). The set \(S\) is said to be \emph{dense} in \((X, d)\) if for every \(a \in X\) there is a sequence \((s_n)_{n\in \mathbb{N}} \subseteq S\) such that
\[
a = \lim_{n \to \infty} s_n.
\]

 A metric space $(X, d)$ is  {\it separable} if and only if there exists a set $A \subseteq X$ 
such that $A$ is dense in $(X, d)$ and $|A| \leq \aleph_0$ holds.

We will use the following auxiliary statement.

\begin{lemma} \label{satur}
Let $ (X, d)$ be a separable metric space and let $ A \subseteq X $ be nonempty. Then the metric space $(A, d|_{A \times A})$ is also separable.
\end{lemma}

See, for example, \cite[p. 216]{Surinder}, Corollary 8.11.

\begin{definition}\label{7jgreoo}
A metric space $(X,d)$ is called \textit{locally finite} if every bounded $A\subseteq X$ is finite.
\end{definition}

\begin{remark}
Definition \ref{7jgreoo} is used, for example, in \cite{Capraro}, \cite{Ost-Ost}, \cite{Ost}, but is not generally accepted.
In Encyclopedia of Distances \cite{Dez-Dez} and many other works, the locally finite spaces $(X,d)$  are required to satisfy Definition~\ref{7jgreoo} and the following condition: there exists $r > 0$ such that the ball $B_r(x)$ contains exactly one point for each $x \in X$.
\end{remark}

The next statement follows directly from Definition~\ref{7jgreoo}.

\begin{lemma}\label{ettpaz86}
Let $(X,d)$ be a locally finite metric space and let $A $ be a subset of $X$. Then the metric space $(A, d|_{A \times A})$ is locally finite.
\end{lemma}

\begin{proof}

It suffices to note that each bounded in $ (A, d|_{A \times A})$ set is also  bounded in $ (X, d). $
\end{proof}

Further, let us recall some definitions and facts from graph theory.

A \textit{graph} is a pair $(V, E)$ consisting of a set $V$ and a set $E$ whose elements are unordered pairs $\{u, v\}$ of different points $u, v \in V$. For a graph $G = (V, E)$, the sets $V = V(G)$ and $E = E(G)$ are called \textit{the set of vertices} and \textit{the set of edges}, respectively. A graph $G$ is \textit{finite} if $V(G)$ is a finite set. If $\{x, y\} \in E(G)$, then the vertices $x$ and $y$ are called \textit{adjacent}. In what follows, we will always assume that $E(G) \cap V(G) = \emptyset$.

Let $G$ be a graph.
A graph \(G_1\) is a \emph{subgraph} of \(G\) if
\[
V(G_1) \subseteq V(G) \quad \text{and} \quad E(G_1) \subseteq E(G).
\]
In this case we will write \(G_1 \subseteq G\). 

If \(\{G_i \colon i \in I\}\) is a family of subgraphs of a graph \(G\), then, by definition, the union \(\bigcup_{i \in I} G_i\) is a subgraph \(G^{*}\) of \(G\) such that 
\[
V(G^{*}) = \bigcup_{i \in I} V(G_i) \quad \text{and} \quad E(G^{*}) = \bigcup_{i \in I} E(G_i).
\]
Similarly, the intersection \(\bigcap_{i \in I} G_i\) is a subgraph \(G_{*}\) of \(G\) with
\begin{equation*}
V(G_{*}) = \bigcap_{i \in I} V(G_i) \quad \text{and} \quad E(G_{*}) = \bigcap_{i \in I} E(G_i).
\end{equation*}

Let \(v\) be a vertex of a graph \(G\), and let $N(v)$ be the set of vertices of $G$ adjacent to $v$,
\begin{equation}\label{edtg}
N(v) = \{ u \in V(G) \colon \{v, u\} \in E(G) \}.
\end{equation}
The {\it degree} of the vertex $v$ is, by definition, the cardinality of the set $N(v)$.  
The degree of $v$ will be denoted as $\delta_G(v)$. Thus, we have
\begin{equation*}
\delta_G(v) = |N(v)|.
\end{equation*}

A \emph{path} is a finite graph \(P\) whose vertices can be numbered without repetitions so that
\begin{equation}\label{e3.3-1}
V(P) = \{x_1, \ldots, x_k\} \quad \text{and} \quad E(P) = \{\{x_1, x_2\}, \ldots, \{x_{k-1}, x_k\}\}
\end{equation}
with \(k \geqslant 2\). We will write \(P = (x_1, \ldots, x_k)\) or \(P = P_{x_1, x_k}\) if \(P\) is a path satisfying \eqref{e3.3-1} and said that \(P\) is a \emph{path joining \(x_1\) and \(x_k\)}. A graph \(G\) is \emph{connected} if for every two distinct vertices of \(G\) there is a path \(P \subseteq G\) joining these vertices.

A finite graph $C$ is a \textit{cycle} if there exists an enumeration of its vertices without repetition such that $V(C) = \{x_1, \ldots, x_n\}$ and
\begin{equation*}
E(C) = \{\{x_1, x_2\}, \ldots, \{x_{n-1}, x_n\}, \{x_n, x_1\}\} \quad \text{with } n \geq 3.
\end{equation*}

\begin{definition}
A connected graph $T$ with $V(T) \neq \emptyset$ and without cycles is called a \textit{tree}.
\end{definition}

We shall say that a tree \(T\) is a \emph{star graph} if there is a vertex \(c \in V(T)\), the \emph{center} of \(T\), such that \(c\) and \(v\) are adjacent for every \(v \in V(T) \setminus \{c\}\).

An infinite graph \(R\) of the form 
\begin{equation}
\label{santr1}
V(R) = \{x_1, x_2, \ldots, x_n, x_{n+1}, \ldots\}, \quad E(R) = \{\{x_1, x_2\}, \ldots, \{x_n, x_{n+1}\}, \ldots\},
\end{equation}
where all \(x_n\) are assumed to be distinct, is called a \emph{ray}. If \eqref{santr1} holds then we will write $R=(x_1,x_2,\ldots,x_n,\ldots)$. It is clear that every ray is a tree. A graph is \emph{rayless} if it contains no rays.

The next proposition is a modification of König's Infinity Lemma \cite{Konig}.

\begin{proposition}\label{p3.2}
Every infinite connected graph has a vertex of infinite degree or contains a ray.
\end{proposition}

A proof of Proposition \ref{p3.2} is provided, for instance, in \cite{Diestel2017}, Proposition~8.2.1.

The following statement is well known in the context of finite trees. A proof for infinite trees can be found in \cite{DK2024DLPSSAG}, Lemma~1.

\begin{lemma}\label{l9.14-1}
In each tree, every two different vertices are connected by exactly one path.
\end{lemma}

In the next definition, we recall the notion of a hull for arbitrary trees.

\begin{definition}\label{d3.6}
Let \(T\) be a tree, and let \(A\) be a nonempty subset of \(V(T)\). A subtree \(H_A\) of the tree \(T\) is the \emph{hull} of \(A\) if \(A \subseteq V(H_A)\) and, for every subtree \(T^{*}\) of \(T\), the tree \(H_A\) is a subtree of \(T^{*}\) whenever \(A \subseteq V(T^*)\).
\end{definition}

Thus, \(H_A\) is the smallest subtree of \(T\) which contains \(A\).

\begin{proposition}\label{p3.7}
Let \(T\) be a tree, \(A\) be a nonempty subset of \(V(T)\) and let \(\mathcal{F}_A\) be the set of all subtrees \(T^{*}\) of \(T\) for which the inclusion \(A \subseteq V(T^{*})\) holds. Then the graph \(\bigcap_{T^{*} \in \mathcal{F}_A} T^{*}\) is the hull of \(A\),
\begin{equation}\label{p3.7:e1}
H_A = \bigcap_{T^{*} \in \mathcal{F}_A} T^{*}.
\end{equation}
\end{proposition}

A proof can be found in \cite{DK2022LTGCCaDUS}.

The following proposition gives a ``constructive'' description of the hulls of sets $A \subseteq V(T)$ for trees $T$ with $|V(T)| \geq 2$.

\begin{proposition}\label{p3.7-1} Let $T$ be a tree and let $A \subseteq V(T)$ contain at least two points. Then the equality
\begin{equation}\label{qrrr331}
H_A = \bigcup_{v \in A \setminus \{u\}} P_{u,v}
\end{equation}
holds for every $u \in A$, where $P_{u,v}$ is the path connecting $u$ and $v$ in $T$.
\end{proposition}

For the proof, see \cite{DV2025JMS}.

Let us introduce the concept of labeled trees.

\begin{definition}\label{d2.4}
A {\it labeled tree} is a pair \((T, l)\), where \(T\) is a tree and \(l\) is a mapping defined on the set \(V(T)\).
\end{definition}
If $(T,l)$ is a labeled tree, then we will say that \(T\) is a \emph{free tree} corresponding to \((T, l)\) and write \(T = T(l)\) instead of \((T, l)\). Moreover, in what follows, we will consider only the nonnegative real-valued labelings \(l\colon V(T)\to \mathbb R^{+}\).

Following~\cite{Dov2020TaAoG}, for arbitrary labeled tree \(T = T(l)\), we define a mapping \(d_l \colon V(T) \times V(T) \to \mathbb R^{+}\) as
\begin{equation}\label{e11.3}
d_l(u, v) = \begin{cases}
0 & \text{if } u = v,\\
\max\limits_{v^{*} \in V(P)} l(v^{*}) & \text{if } u \neq v,
\end{cases}
\end{equation}
where \(P\) is a path joining \(u\) and \(v\) in \(T\).

\begin{theorem}\label{t11.9}
Let \(T = T(l)\) be a labeled tree. The mapping \(d_l\) is an ultrametric  on the set $V(T)$ if and only if the inequality
\begin{equation}\label{t11.9:e1}
\max\{l(u), l(v)\} > 0
\end{equation}
holds for every \(\{u, v\} \in E(T)\).
\end{theorem}

A proof of Theorem \ref{t11.9} can be obtained by simple modification of the proof of Proposition~3.2 in \cite{Dov2020TaAoG}.

In what follows we shall say that a labeling $l \colon V(T) \to \mathbb{R}^+$ is {\it non-degenerate} if the inequality
\begin{equation*}
\max \{ l(u), l(v) \} > 0
\end{equation*}
holds for every $\{u, v\} \in E(T)$.

\section{Separability of Generated Ultrametrics}\label{sec3}

 The goal of this section is to describe the structure of free trees $T$ which admit non-degenerate labelings 
$l \colon V(T) \to \mathbb{R}^+$ generating separable ultrametric spaces $(V(T), d_l)$.

In what follows, we will say that a set $A$ is {\it countable} if the inequality $|A| \leq \aleph_0$ holds.

\begin{lemma}\label{wseet} Let $T$ be a tree. If the inequality
\begin{equation}
\label{umiiw}
\delta_T(v) \leq \aleph_0 
\end{equation}
holds for every $v \in V(T)$, then the set $V(T)$ is countable, 
\begin{equation}
\label{efikki}
|V(T)| \leq \aleph_0. 
\end{equation}
\end{lemma}

\begin{proof} 

Let $v^*$ be a fixed vertex of $T$. Then we can define a sequence $(N_j(v^*))_{j \in \mathbb{N}}$ by
\begin{equation*}
N_1(v^*): = N(v^*), \quad N_2(v^*): = \bigcup_{u \in N_1(v^*)} N(u),
\end{equation*}
and, for every integer $j \geq 2$,
\begin{equation}
\label{edik}
N_{j+1}(v^*) := \bigcup_{u \in N_j(v^*)} N(u),
\end{equation}
where $N(v^*)$ and $N(u)$ are defined as in \eqref{edtg}.

Since $T$ is a connected graph, the equality 
\begin{equation}
\label{umidsc}
V(T) = \bigcup_{j \in \mathbb{N}} N_j(v^*)
\end{equation}
holds.
Let \eqref{umiiw} hold for every $v \in V(T)$. Inequality \eqref{umiiw} can be rewritten as
\begin{equation*}
|N(v)| \leq \aleph_0. 
\end{equation*} Hence, by \eqref{edik}, every $N_j(v^*)$ is a countable union of countable sets. Thus the inequality
\begin{equation*}
|N_j(v^*)| \leq \aleph_0
\end{equation*}
holds for every $j \in \mathbb{N}$.
Similarly, the last inequality and equality \eqref{umidsc} imply inequality \eqref{efikki}.

\end{proof}

The next theorem is the main result of this section.

\begin{theorem}\label{ggewj77}
Let $T$ be a tree. Then the following statements are equivalent:
\begin{enumerate}[label=\textit{(\roman*)}, left=0pt]
    \item The vertex set of $T$ is  countable.

   \item For every non-degenerate labeling $l \colon V(T) \to \mathbb{R}^+$, the ultrametric space $(V(T), d_l)$ is separable.
    
    \item There exists a labeling $l \colon V(T) \to \mathbb{R}^+$ such that $(V(T), d_l)$ is a separable ultrametric space.

\end{enumerate}
\end{theorem}

\begin{proof} 

$(i)$ $\Rightarrow$ $(ii)$. 
This implication follows directly from the definition of separability.

$(ii)$ $\Rightarrow$ $(iii)$.
The implication is true because every tree $T$ admits a non-degenerate labeling 
$l \colon V(T) \to \mathbb{R}^+$.

$(iii)$ $\Rightarrow$ $(i)$. Suppose, to the contrary, that there exists a non-degenerate labeling $l \colon V(T) \to \mathbb{R}^+$ such that $(V(T), d_l)$ is separable, but the inequality
\begin{equation*}
|V(T)| > \aleph_0
\end{equation*}
holds. Then, by Lemma \ref{wseet}, there is $v^* \in V(T)$ such that 
\begin{equation}
\label{tag2}
\delta_T(v^*) > \aleph_0. 
\end{equation}
Let us suppose first that $l(v^*) > 0$. Then equality~\eqref{e11.3} gives us
\begin{equation}\label{okjfe}
d_l(u,v) = \max \{ l(v), l(v^*), l(u) \} \geq  l(v^*)  > 0
\end{equation}
for all different $u, v \in N(v^*)$.

Hence, the ultrametric space $  (N(v^*), d_l|_{N(v^*) \times N(v^*)})$   is a discrete subspace of $(V(T), d_l)$.
Moreover, this subspace is a separable subspace of $(V(T), d_l)$ by Lemma~\ref{satur}.
In every non-empty discrete metric space $(X,d)$, every dense subset of $(X,d)$ coincides with
$X$. Consequently, the inequality
\begin{equation*}
|N(v^*)| \leq \aleph_0
\end{equation*}
holds because $  (N(v^*), d_l|_{N(v^*) \times N(v^*)})$ is separable.
The last inequality and the equality
\begin{equation*}
|N(v^*)| = \delta_T(v^*)
\end{equation*}
imply that inequality \eqref{tag2} is false, a contradiction to the supposition.

Let us now consider the case when the equality 
\begin{equation}\label{wsroii}
    l(v^*) = 0
\end{equation}
holds.

Since $l \colon V(T) \to \mathbb{R}^+$ is non-degenerate, equality \eqref{wsroii} 
implies the inequality 
\begin{equation}\label{e774}
    l(v) > 0
\end{equation}
for every $v \in N(v^*)$.

For each $m \in \mathbb{N}$, we denote by $N^m(v^*)$ the set of all $v \in N(v^*)$ satisfying the inequality
\begin{equation*}
l(v) \geq \frac{1}{m}.
\end{equation*}

Since \eqref{e774} holds for every $v \in N(v^*)$, we obtain the equality
\begin{equation*}
N(v^*) = \bigcup_{m \in \mathbb{N}} N^m(v^*).
\end{equation*}

Using the last equality and inequality~\eqref{tag2}, we can find $m_0 \in \mathbb{N}$ such that
\begin{equation}\label{rvhww33}
|N^{m_0}(v^*)| > \aleph_0. 
\end{equation}

Indeed, otherwise $N(v^*)$ is a countable union of countable sets $N^m(v^*)$ and hence also is countable,
\begin{equation*}
    |N(v^*)| \leq \aleph_0,
\end{equation*}
contrary to \eqref{tag2}.

Let $m_0 \in \mathbb{N}$ satisfy \eqref{rvhww33}. Then $N^{m_0}(v^*)$ is an uncountable subset of $V(T)$.
Then, arguing as in proof of \eqref{okjfe}, we obtain
\begin{equation*}
d_l(u, v) \geq \frac{1}{m_0}
\end{equation*}
for all distinct $u, v \in N^{m_0}(v^*)$.

Thus, $(N^{m_0}(v^*), d_l|_{N^{m_0}(v^*) \times N^{m_0}(v^*)})$ is an uncountable discrete subspace of $(V(T), d_l)$.
As was proven above in the analysis of the case $l(v^*) > 0$, the last statement leads to a contradiction with inequality~\eqref{tag2}.

This completes the proof.

\end{proof}

Theorem~\ref{ggewj77} and Lemma \ref{wseet} give us the following corollary.

\begin{corollary}\label{zaz}
    Let $T$ be a tree. Then the following statements are equivalent:

\begin{enumerate}[label=\textit{(\roman*)}, left=0pt]
    \item There exists a labeling $l \colon V(T) \to \mathbb{R}^+$ such that $(V(T), d_l)$ is a separable ultrametric space.

\item For every star graph $S \subseteq T$ there exists a non-degenerate labeling $l_S \colon V(S) \to \mathbb{R}^+$ such that the ultrametric space $(V(S), d_{l_S})$ is separable.

\end{enumerate}
\end{corollary}

The following corollary describes the labeled star graphs generating separable ultrametric spaces.

\begin{corollary}\label{dfuj247-1}
Let $S = S(l)$ be a star graph with non-degenerate labeling $l_S \colon V(S) \to \mathbb{R}^{+}$. Then the following statements are equivalent:
\begin{enumerate}[label=\textit{(\roman*)}, left=0pt] 
\item The ultrametric space $(V(S), d_{l_S})$ is separable.

\item The set 
\begin{equation}\label{rvgujq-1}
W_\varepsilon := \{v \in V(S) \colon l_S(v) \leq \varepsilon\}
\end{equation}
is countable for every $\varepsilon > 0$.
\end{enumerate}
\end{corollary}

\begin{proof}

$(i) \Rightarrow (ii)$. Let (i) hold. Then $V(S)$ is countable by Theorem~\ref{ggewj77}. Hence, for every $\varepsilon > 0$, the set $W_\varepsilon$ is also countable as a subset of a countable set.

$(ii) \Rightarrow (i)$. Let (ii) hold and let $(\varepsilon_n)_{n \in \mathbb{N}}$ be a sequence of positive numbers such that
\begin{equation*}
\lim_{n \to \infty} \varepsilon_n = \infty.
\end{equation*}

The last limit relation and \eqref{rvgujq-1} imply the equality
\begin{equation}
\bigcup_{n \in \mathbb{N}} W_{\varepsilon_n} = V(S).
\end{equation}

By statement $(ii)$, each $W_{\varepsilon_n}$ is countable. Hence $V(S)$ also is countable as a countable union of countable sets. Hence $(V(S), d_{l_S})$ is separable by Theorem~\ref{ggewj77}.

This completes the proof.
\end{proof}

\section{Local Finiteness of Generated Ultrametrics}\label{sec4}

Let us now turn to locally finite ultrametric spaces $(V(T), d_l)$.

\begin{lemma}\label{ljhg35}
Let $T = T(l)$ be a labeled tree with non-degenerate labeling $l \colon V(T) \to \mathbb{R}^+$, let $A$ be a nonempty bounded subset of the ultrametric space $(V(T), d_l)$. 
Then the vertex set of hull $H_A$ is also a bounded subset of $(V(T), d_l)$.
\end{lemma}

\begin{proof}

If $A$ is a singleton set, then $V(H_A)$ is also a singleton set and, consequently, $V(H_A)$ is bounded.

Suppose that $A$ contains at least two points and let $u$ be an arbitrary point in $A$.
Then, by Proposition~\ref{p3.7-1}, the equality
\begin{equation}\label{qrrr331-s}
H_A = \bigcup_{v \in A \setminus \{u\}} P_{u,v}
\end{equation}
holds. Consider an arbitrary point $w \in V(H_A)$. Equality \eqref{qrrr331-s} implies that there is $v \in A \setminus \{u\}$ such that the point $w$ is a vertex of the path $P_{u,v}$.

Consequently the inequality
\begin{equation}\label{stad1}
d_l(u, w) \leq d_l(u, v)
\end{equation}
holds by definition of the ultrametric $d_l$. Since $A$ is a bounded subset of $(V(T), d_l)$, there is $r > 0$ such that
\begin{equation}\label{stad2}
d_l(u, v) \leq r 
\end{equation}
for all $v \in V(H_A)$. Thus, the set $V(H_A)$ is a subset of the closed ball $B_r(u)$ and, consequently, $V(H_A)$ is bounded as required.

\end{proof}

The following proposition gives a complete description of labelings $l \colon V(T) \to \mathbb{R}^+$ generating locally finite ultrametric spaces $(V(T), d_l)$ for the case where trees $T$ are rays.

\begin{proposition}\label{52dcij}
Let $l_R\colon V(R) \to \mathbb{R}^+$ be a non-degenerate labeling on the vertex set of a ray $R = (v_1, v_2, \ldots, v_n, \ldots)$. Then the following statements are equivalent:

\begin{enumerate}[label=\textit{(\roman*)}, left=0pt]
\item The ultrametric space $(V(R), d_{l_R})$ is locally finite.

\item The limit relation
\begin{equation}\label{tghuaa}
\limsup_{n \to \infty} l_R(v_n) = \infty 
\end{equation}
holds.
\end{enumerate}
\end{proposition}

\begin{proof}

$ (i)$ $\Rightarrow$ $(ii)$. 
Suppose, that the limit superior in \eqref{tghuaa} is a finite 
\begin{equation}
\limsup_{n \to \infty} l_R(v_n) = k < \infty.
\end{equation}
Then for each $\varepsilon > 0$, there is $m \in \mathbb{N}$ such that  
\begin{equation*}
l_R(v_n) \leq k + \varepsilon 
\end{equation*}
for all $n \geq m.$ It implies the inequality  
\begin{equation*}
l_R(v_n) \leq \max\{l_R(v_1), \ldots, l_R(v_{m-1}), k + \varepsilon\}
\end{equation*}
for all $n \in \mathbb{N}$. Hence the inequality
\begin{equation}\label{oi45dv}
\sup_{n \in \mathbb{N}} l_R(v_n) < \infty
\end{equation}
holds. Now using \eqref{e11.3} and \eqref{oi45dv}, we obtain
\begin{equation*}
d_{l_R}(u, w) \leq \sup_{n \in \mathbb{N}} l_R(v_n) < \infty
\end{equation*}
for all $u, w \in (V(R),d_{l_R})$. Consequently, $(V(R),d_{l_R})$ is bounded and infinite, contrary to the supposition.

$(ii)$ $\Rightarrow$ $(i)$. 
Let $A \subseteq V(R)$ be an arbitrary bounded subset.
We must prove that $A$ is finite.
Since $A$ is bounded, there is a constant $c > 0$ such that
\begin{equation}\label{jtfdqw11}
d_{l_R}(u, w) < c 
\end{equation}
for all $u, w \in A$. Using \eqref{tghuaa} we can find an infinite subsequence $(u_{n_k})_{k \in \mathbb{N}}$ of the sequence $(v_n)_{n \in \mathbb{N}}$ such that $n_1 \geq 2$ and
\begin{equation}\label{onhwer12}
l_R(v_{n_k}) \geq c
\end{equation}
for each $k \in \mathbb{N}$. Let us define a sequence $(P_k)_{k \in \mathbb{N}}$ of paths in $R$ by the rule: 
\begin{equation*}
P_1= (v_1, \ldots, v_{n_1})
\end{equation*}
is
the path joining $v_1$ and $v_{n_1}$, and 
\begin{equation*}
P_k = (v_{n_{k-1}}, \ldots, v_{n_k})
\end{equation*}
for every $k \geq 2$.

We claim that a sequence $(P_k)_{k \in \mathbb{N}}$ contains a path $P_{k_0}$ such that
\begin{equation}\label{ettbu531}
A \subseteq V(P_{k_0}). 
\end{equation}

Indeed if the last claim is false, then there exist two distinct paths $P_{k_1}$ and $P_{k_2}$ satisfying the conditions
\begin{equation}\label{ettbu532}
A \cap V(P_{k_1}) \neq \emptyset \neq A \cap V(P_{k_2}), 
\end{equation}
and
\begin{equation}\label{ettbu533}
A \cap V(P_{k_1}) \neq A \cap V(P_{k_2}). 
\end{equation}

Using \eqref{ettbu532} and \eqref{ettbu533}, we can find $a_1, a_2 \in A$ such that
\begin{equation}\label{ettbu534}
a_1 \in V(P_{k_1}) \setminus V(P_{k_2}) \quad \text{and} \quad a_2 \in V(P_{k_2}) \setminus V(P_{k_1}). 
\end{equation}
Without loss of generality, we assume that the inequality
\begin{equation}\label{ettbu535}
k_1 < k_2 
\end{equation}
holds. 
Let $P_{a_1,a_2}$ be a path joining $a_1$ and $a_2$ in $R$. 
Membership relations \eqref{ettbu534} and inequality \eqref{ettbu535} imply 
\begin{equation}\label{ettbu536}
v_{n_{k_1}} \in V(P_{a_1,a_2}). 
\end{equation}
Hence
\begin{equation}\label{ettbu537}
d_{l_R}(a_1, a_2) = \max_{v \in P_{a_1,a_2}} l(v)\geq l(v_{n_{k_1}})
\end{equation}
holds by \eqref{e11.3}, \eqref{ettbu536}.
Now, using \eqref{onhwer12} and \eqref{ettbu537}, we get the inequality $d_{l_R}(a_1, a_2) \geq c$,
which contradicts \eqref{jtfdqw11}.

Thus there is a path $P_{k_0}$ satisfying inclusion \eqref{ettbu531} that implies
\begin{equation*}
|A| \leq |V(P_{k_0})| < \infty.
\end{equation*}

This completes the proof.

\end{proof}

The following proposition describes the labeled star graphs generating locally  finite ultrametric spaces.

\begin{proposition}\label{dfuj247}
Let $S = S(l)$ be a star graph with non-degenerate labeling $l_S \colon V(S) \to \mathbb{R}^{+}$. Then the following statements are equivalent:
\begin{enumerate}[label=\textit{(\roman*)}, left=0pt] 
\item The ultrametric space $(V(S), d_{l_S})$ is locally finite.

\item The set 
\begin{equation}\label{rvgujq}
W_\varepsilon := \{v \in V(S) \colon l_S(v) \leq \varepsilon\}
\end{equation}
is finite for every $\varepsilon > 0$.
\end{enumerate}
\end{proposition}

\begin{proof}

$(i)$ $\Rightarrow$ $(ii)$. Suppose, for contradiction, there exists $\varepsilon_0 > 0$ such that the set $W_{\varepsilon_0}$ given in~\eqref{rvgujq} is infinite.

Write $H_{W_{\varepsilon_0}}$ for the hull of the set $W_{\varepsilon_0}$ and denote by $c$ the center of the star graph $S$.
Then $H_{W_{\varepsilon_0}}$ is a subtree of the star graph $S$. Since every connected subgraph of $S$ is also a star graph with the same center, $H_{W_{\varepsilon_0}}$ is a star graph with the center $c$.

The vertex set of $H_{W_{\varepsilon_0}}$ is the set $W_{\varepsilon_0} \cup \{c\}$,
\begin{equation}\label{kkyrew62}
V(H_{W_{\varepsilon_0}}) = W_{\varepsilon_0} \cup \{c\}.
\end{equation}

Equality \eqref{rvgujq} with $\varepsilon = \varepsilon_0$ and equality \eqref{kkyrew62}  give us the inequality
\begin{equation*}
l_S(v) \leq \max\{\varepsilon_0, l_S(c)\}
\end{equation*}
for every $v \in V(H_{W_{\varepsilon_0}})$.
The last inequality and equality \eqref{e11.3} with $T = S$ and $l = l_S$ show that $V(H_{W_{\varepsilon_0}})$ is a bounded subset of the ultrametric space $(V(S), d_{l_S})$.
Since $(V(S), d_{l_S})$ is locally bounded, the boundedness of the set 
$V(H_{W_{\varepsilon_0}})$ implies the finiteness of this set. Consequently, the set $W_{\varepsilon_0}$ is also finite, which contradicts the definition of the number $\varepsilon_0$. 

Thus, the set $W_\varepsilon$ is finite for each $\varepsilon>0$.

$(ii)$ $\Rightarrow$ $(i)$. Let $A$ be an arbitrary bounded subset of the ultrametric space $(V(S), d_{l_S})$. We must prove the inequality
\begin{equation}\label{qswssej1}
|A| < \infty.
\end{equation}

Since $A$ is bounded, there is a positive number $\varepsilon_1$ such that
\begin{equation}\label{kuhgs}
d_{l_S}(u, v) \leq \varepsilon_1 < \infty 
\end{equation}
for all $u, v \in A.$
Using \eqref{e11.3}
and the definition of the star graphs we can prove the equality
\begin{equation}\label{olidhh}
d_{l_S}(u,v) = \max\{l_S(u), l_S(c), l_S(v)\} 
\end{equation}
for all different $u,v \in V(S)$, where $c$ is the center of the star graph $S$. Now \eqref{kuhgs} and \eqref{olidhh} imply the inequality
\begin{equation*}
l_S(v) \leq \varepsilon_1 
\end{equation*}
for each $v \in A$.
The last inequality and formula \eqref{rvgujq} show that the inclusion
\begin{equation}\label{aaa33}
A \subseteq W_{\varepsilon_1} 
\end{equation}
holds. By statement $(ii)$, the set $W_{\varepsilon}$ is finite for each $\varepsilon > 0$.
Thus, \eqref{qswssej1} follows from \eqref{aaa33}.

This completes the proof.

\end{proof}

The next theorem is an analog of König's Infinity Lemma for locally finite ultrametric spaces $(V(T), d_l)$.

\begin{theorem}\label{dster}
Let $T = T(l)$ be a labeled tree with non-degenerate labeling $l \colon V(T) \to \mathbb{R}^+$. Then the following statements are equivalent:

\begin{enumerate}[label=\textit{(\roman*)}, left=0pt] 
\item The ultrametric space $(V(T), d_l)$ is locally finite.

\item The ultrametric spaces $(V(R), d_l|_{V(R) \times V(R)})$ and $(V(S), d_l|_{V(S) \times V(S)})$ are locally finite for all rays $R \subseteq T$ and all star graphs $S \subseteq T$.
\end{enumerate}
\end{theorem}

\begin{proof}

$(i)$ $\Rightarrow$ $(ii)$. The truth of this implication follows from Lemma~\ref{ettpaz86}.

$(ii)$ $\Rightarrow$ $(i)$. Let $A$ be a nonempty bounded subset of $V(T)$.
We must prove that $A$ is a finite subset of $V(T)$.

Let $H_A$ be the hull of $A$. Proposition~\ref{p3.7-1} implies that $A$ is a finite set if and only if $H_A$ is a finite graph. 

By Proposition~\ref{p3.2} the graph $H_A$ is infinite if and only we have
\begin{equation}\label{tau1}
R \subseteq H_A 
\end{equation}
for some ray $R$, or 
\begin{equation}\label{tau2}
S \subseteq H_A
\end{equation}
for some  infinite star graph $S$.

Suppose first that \eqref{tau1} holds. Then the ultrametric space 
$
(V(R), d_l|_{V(R) \times V(R)})
$
is locally finite by statement $(ii)$. Hence $V(R)$ is unbounded subset of $(V(T), d_l)$ by Proposition~\ref{52dcij}.

Furthermore, since $A$ is bounded in $(V(T), d_l)$, the set $V(H_A)$ is also bounded in $(V(T), d_l)$ by Lemma~\ref{ljhg35}. Thus, we have obtained a contradiction: the unbounded set $V(R)$ is a subset of the bounded set $V(H_A)$. 

Therefore, there are no rays $R$ satisfying inclusion \eqref{tau1}.

By arguing in a similar way and using Proposition~\ref{dfuj247} instead of Proposition~\ref{52dcij}, one can show that inclusion \eqref{tau2} is also impossible for infinite star graphs $S$.

Thus, $H_A$ is a finite graph.

This completes the proof.
\end{proof}

\begin{example}
 Let $T$ be the tree containing the ray $R = (v_1, v_2, \ldots, v_n, \ldots)$ with labeling $l_R \colon V(T) \to \mathbb{R}^+$ satisfying $l_R(v_n) = n$ for each $n \in \mathbb{N}$ (see Figure~\ref{fig9-11}).
If $l \colon V(T) \to \mathbb{R}^+$ is a non-degenerate labeling such that $l_R$ is a restriction of $l$ on the set $V(R)$, then the ultrametric space $(V(T), d_l)$ is locally finite by Proposition~\ref{52dcij} and Theorem~\ref{dster}.

\end{example}

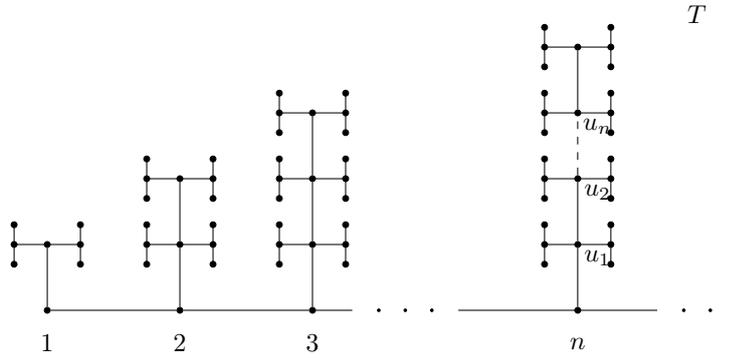
\begin{figure}[!ht]
\centering
\resizebox{0.75\textwidth}{!}{
\begin{tikzpicture}[scale=0.9, every node/.style={circle, draw, fill=black, inner sep=0.8pt}]

\node (v11) at (0,0) {};
\node (v12) at (0,1) {};
\draw (v11) -- (v12);
\draw (-0.5,1) -- (0.5,1);
\draw (-0.5,1.3) -- (-0.5,0.7);
\draw (0.5,1.3) -- (0.5,0.7);
\node at (-0.5,1.3) {};
\node at (-0.5,0.7) {};
\node at (0.5,1.3) {};
\node at (0.5,0.7) {};
\node at (-0.5,1) {};
\node at (0.5,1) {};
\node[draw=none, fill=none] at (0,-0.5) {1};

\node (v21) at (2,0) {};
\node (v22) at (2,1) {};
\node (v23) at (2,2) {};
\draw (v21) -- (v22) -- (v23);
\foreach \y in {1,2} {
  \draw (1.5,\y) -- (2.5,\y);
  \draw (1.5,\y+0.3) -- (1.5,\y-0.3);
  \draw (2.5,\y+0.3) -- (2.5,\y-0.3);
  \node at (1.5,\y+0.3) {};
  \node at (1.5,\y-0.3) {};
  \node at (2.5,\y+0.3) {};
  \node at (2.5,\y-0.3) {};
  \node at (1.5,\y) {};
  \node at (2.5,\y) {};
}
\node[draw=none, fill=none] at (2,-0.5) {2};

\node (v31) at (4,0) {};
\node (v32) at (4,1) {};
\node (v33) at (4,2) {};
\node (v34) at (4,3) {};
\draw (v31) -- (v32) -- (v33) -- (v34);
\foreach \y in {1,2,3} {
  \draw (3.5,\y) -- (4.5,\y);
  \draw (3.5,\y+0.3) -- (3.5,\y-0.3);
  \draw (4.5,\y+0.3) -- (4.5,\y-0.3);
  \node at (3.5,\y+0.3) {};
  \node at (3.5,\y-0.3) {};
  \node at (4.5,\y+0.3) {};
  \node at (4.5,\y-0.3) {};
  \node at (3.5,\y) {};
  \node at (4.5,\y) {};
}
\node[draw=none, fill=none] at (4,-0.5) {3};

\draw (v11) -- (v21) -- (v31);
\draw (4,0) -- (4.6,0); 
\foreach \x in {5.0,5.4,5.8} {
  \node[fill=black,inner sep=0.3pt] at (\x,0) {};
}
\draw (6.2,0) -- (8,0); %

\node (vn0) at (8,0) {};
\node (vn1) at (8,1) {};
\node (vn2) at (8,2) {};
\node (vn3) at (8,3) {};
\draw[dashed] (vn2) -- (vn3);
\node (vn4) at (8,4) {};
\draw (vn0) -- (vn1) -- (vn2);
\draw (vn3) -- (vn4);

\foreach \y in {1,2,3,4} {
  \draw (7.5,\y) -- (8.5,\y);
  \draw (7.5,\y+0.3) -- (7.5,\y-0.3);
  \draw (8.5,\y+0.3) -- (8.5,\y-0.3);
  \node at (7.5,\y+0.3) {};
  \node at (7.5,\y-0.3) {};
  \node at (8.5,\y+0.3) {};
  \node at (8.5,\y-0.3) {};
  \node at (7.5,\y) {};
  \node at (8.5,\y) {};
}

\node[draw=none, fill=none] at (8,-0.5) {$n$};
\node[draw=none, fill=none] at (8.3,0.8) {$u_1$};
\node[draw=none, fill=none] at (8.3,1.8) {$u_2$};
\node[draw=none, fill=none] at (8.3,2.8) {$u_n$};
\node[draw=none, fill=none] at (9.8,4.5) {$T$};

\draw (8,0) -- (9.2,0); %
\foreach \x in {9.6,10.0,10.4} {
  \node[fill=black,inner sep=0.3pt] at (\x,0) {};
}

\end{tikzpicture}
}
\caption{The tree $T$ does not contain any vertex of infinite degree and the symmetric difference 
of vertex sets $V(R_1)$, $V(R_2)$ is
finite for any two rays $R_1 \subseteq T$, $R_2 \subseteq T$.}
\label{fig9-11}
\end{figure}

\begin{example}\label{lkfa}
Let $T$ be the tree containing the vertex $v$ with $\deg_T(c) = \aleph_0$,  let $S$ be the star graph induced in $T$ by $c$ and all vertices adjacent to $c$,
and let
$
l_S \colon V(S) \to \mathbb{R}^+ 
$
be an injective labeling such that $l_S(V(S)) = \mathbb{N}$ (see Figure~\ref{fig10-11}).
If $l \colon V(T) \to \mathbb{R}^+$ is a non-degenerate labeling such that $l_S$ is a restriction of $l$ on the set $V(S)$, then the ultrametric space $(V(T), d_l)$ is locally finite by Proposition~\ref{dfuj247} and Theorem~\ref{dster}.
\end{example}

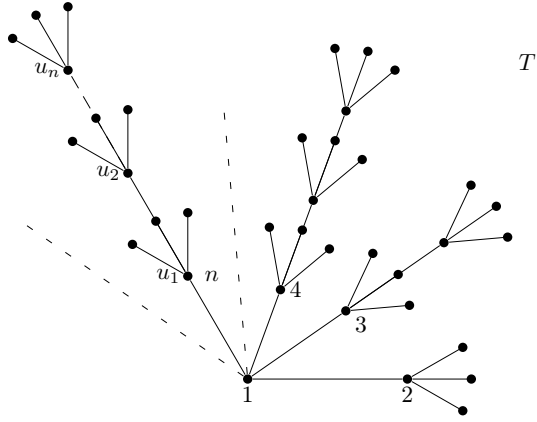
\begin{figure}[!ht]
\centering
\resizebox{0.55\textwidth}{!}{
\begin{tikzpicture}[
    scale=1.2,
    every node/.style={
        circle,
        fill=black,
        minimum size=4pt,
        inner sep=0pt
    }
]

\def\delta{30}

\node (1) at (0,0) [label=below:1] {};

\def\angleA{0}
\coordinate (dir2) at ({cos(\angleA)}, {sin(\angleA)});
\coordinate (leaf2a1) at ({cos(\angleA + \delta)}, {sin(\angleA + \delta)});
\coordinate (leaf2a2) at ({cos(\angleA - \delta)}, {sin(\angleA - \delta)});
\node (2a) at ($(1) + 2*(dir2)$) [label=below:2] {};
\draw (1) -- (2a);
\draw (2a) -- ($(2a) + 0.8*(dir2)$) node {};
\draw (2a) -- ($(2a) + 0.8*(leaf2a1)$) node {};
\draw (2a) -- ($(2a) + 0.8*(leaf2a2)$) node {};

\def\angleB{35}
\coordinate (dir3) at ({cos(\angleB)}, {sin(\angleB)});
\coordinate (leaf3a1) at ({cos(\angleB + \delta)}, {sin(\angleB + \delta)});
\coordinate (leaf3a2) at ({cos(\angleB - \delta)}, {sin(\angleB - \delta)});
\node (3a) at ($(1) + 1.5*(dir3)$) [label={[yshift=-2mm]right:3}] {};
\node (3b) at ($(1) + 3*(dir3)$) {};
\draw (1) -- (3a) -- (3b);

\draw (3a) -- ($(3a) + 0.8*(dir3)$) node {};
\draw (3a) -- ($(3a) + 0.8*(leaf3a1)$) node {};
\draw (3a) -- ($(3a) + 0.8*(leaf3a2)$) node {};

\draw (3b) -- ($(3b) + 0.8*(dir3)$) node {};
\draw (3b) -- ($(3b) + 0.8*(leaf3a1)$) node {};
\draw (3b) -- ($(3b) + 0.8*(leaf3a2)$) node {};

\def\angleC{70}
\coordinate (dir4) at ({cos(\angleC)}, {sin(\angleC)});
\coordinate (leaf4a1) at ({cos(\angleC + \delta)}, {sin(\angleC + \delta)});
\coordinate (leaf4a2) at ({cos(\angleC - \delta)}, {sin(\angleC - \delta)});
\node (4a) at ($(1) + 1.2*(dir4)$) [label=right:4] {};
\node (4b) at ($(1) + 2.4*(dir4)$) {};
\node (4c) at ($(1) + 3.6*(dir4)$) {};
\draw (1) -- (4a) -- (4b) -- (4c);

\draw (4a) -- ($(4a) + 0.8*(dir4)$) node {};
\draw (4a) -- ($(4a) + 0.8*(leaf4a1)$) node {};
\draw (4a) -- ($(4a) + 0.8*(leaf4a2)$) node {};

\draw (4b) -- ($(4b) + 0.8*(dir4)$) node {};
\draw (4b) -- ($(4b) + 0.8*(leaf4a1)$) node {};
\draw (4b) -- ($(4b) + 0.8*(leaf4a2)$) node {};

\draw (4c) -- ($(4c) + 0.8*(dir4)$) node {};
\draw (4c) -- ($(4c) + 0.8*(leaf4a1)$) node {};
\draw (4c) -- ($(4c) + 0.8*(leaf4a2)$) node {};

\def\angleD{95}
\coordinate (dirD) at ({cos(\angleD)}, {sin(\angleD)});
\draw[dashed, dash pattern=on 3pt off 7pt] (1) -- ($(1) + 3.5*(dirD)$);

\def\angleN{120}
\coordinate (dirn) at ({cos(\angleN)}, {sin(\angleN)});
\coordinate (leafna1) at ({cos(\angleN + \delta)}, {sin(\angleN + \delta)});
\coordinate (leafna2) at ({cos(\angleN - \delta)}, {sin(\angleN - \delta)});
\node (u1) at ($(1) + 1.5*(dirn)$) [label=left:$u_1$] {};
\node at ($(u1) + (0.3,0)$) [draw=none, fill=none] { $n$}; 
\node (u2) at ($(1) + 3*(dirn)$) [label=left:$u_2$] {};
\node (un) at ($(1) + 4.5*(dirn)$) [label=left:$u_n$] {};
\draw (1) -- (u1) -- (u2);
\draw[dashed, dash pattern=on 6pt off 4pt] (u2) -- (un); %

\foreach \u in {u1,u2,un} {
    \draw (\u) -- ($(\u) + 0.8*(dirn)$) node {};
    \draw (\u) -- ($(\u) + 0.8*(leafna1)$) node {};
    \draw (\u) -- ($(\u) + 0.8*(leafna2)$) node {};
}

\def\angleE{145}
\coordinate (dirE) at ({cos(\angleE)}, {sin(\angleE)});
\draw[dashed, dash pattern=on 3pt off 7pt] (1) -- ($(1) + 3.5*(dirE)$);

\node[draw=none, fill=none] at (3.5,4) {$T$};

\end{tikzpicture}
}
\caption{The tree 
$T$ is rayless and contains exactly one vertex of infinite degree.}
\label{fig10-11}
\end{figure}

The next theorem describes the trees $T$ which admit labelings 
$l \colon V(T) \to \mathbb{R}^+$ generating locally finite ultrametric spaces $(V(T), d_l)$.

\begin{theorem}\label{trees11}
Let $T$ be a tree. Then the following statements are equivalent:
\begin{enumerate}[label=\textit{(\roman*)}, left=0pt]
    \item The vertex set of $T$ is countable.

\item There is a non-degenerate labeling $l \colon V(T) \to \mathbb{R}^+$ for which the ultrametric space $(V(T), d_l)$ is locally finite.

\end{enumerate}
\end{theorem}

\begin{proof}

$(i) \Rightarrow (ii)$. 
If $T$ is a finite tree, then for every non-degenerate $l\colon V(T) \to \mathbb{R}^+$ the ultrametric space $(V(T), d_l)$ 
is finite and, consequently, locally finite.

Suppose that $|V(T)| = \aleph_0$ holds. Then there is a bijection 
$
f\colon V(T) \to \mathbb{N}
$
and, consequently, we can consider a labeling $l\colon  V(T) \to \mathbb{R}^+$ such that 
\begin{equation*}
l(v) = f(v)
\end{equation*}
for each $v \in V(T)$. Then, using Proposition~\ref{52dcij}, it is easy to prove that, for every ray $R \subseteq T$, the ultrametric space $(V(R), d_l|_{V(R) \times V(R)})$ is a locally finite subspace of $(V(T), d_l)$.

Analogously, Proposition~\ref{dfuj247} implies that $(V(S), d_l|_{V(S) \times V(S)})$ is locally finite for every infinite star graph $S \subseteq T$. 

Hence $(V(T), d_l)$ is locally finite by Theorem~\ref{dster}.

$(ii) \Rightarrow (i)$. 
To prove the inequality 
\begin{equation}\label{baas9}
|V(T)| \leq \aleph_0,
\end{equation}
fix a point $v \in V(T)$ and represent $V(T)$ as the union
\begin{equation*}
V(T) = \bigcup_{n=1}^{\infty} B_{r_n}(v),
\end{equation*}
where $B_{r_n}(v)$ is the closed ball with the radius $r_n = n$ and center $v$,
\begin{equation*}
B_{r_n}(v) := \{u \in V(T) \mid d_l(u,v) \leq n \}
\end{equation*}
 for every $n \in \mathbb{N}$. Since all $B_{r_n}$ are bounded and $(V(T), d_l)$ is locally finite, $V(T)$ is countable as a countable union of finite sets.

This completes the proof.

\end{proof}

Theorems \ref{ggewj77} and \ref{trees11} give us the following corollary.

\begin{corollary}
 Let $T$ be a tree. Then the following statements are equivalent:

 \begin{enumerate}[label=\textit{(\roman*)}, left=0pt]
\item The ultrametric space $(V(T), d_l)$ is separable for every non-degenerate labeling $l \colon V(T) \to \mathbb{R}^+$.

\item There is a non-degenerate labeling $l \colon V(T) \to \mathbb{R}^+$ for which the ultrametric space $(V(T), d_l)$ is locally finite.

\item There is a non-degenerate labeling $l \colon V(T) \to \mathbb{R}^+$ for which the ultrametric space $(V(T), d_l)$ is separable.

\end{enumerate}

\end{corollary}

\begin{example}
 Let $l_R \colon V(R) \to \mathbb{R}^+$
be a non-degenerate labeling on a ray $R = (v_1, v_2, \ldots, v_n, \ldots)$
such that
\begin{equation*}
\lim_{n \to \infty} l_R(v_n) = 0.
\end{equation*}

Then $(V(R), d_{l_R})$ is separable, by Theorem~\ref{ggewj77}, but not locally finite, by Proposition~\ref{52dcij}.
\end{example}


\subsection*{Funding}

First author was supported by grant 359772 of the Academy of Finland.

\subsection*{Conflict of interest}

The authors declare no conflict of interest.






\bibliography{bib2020.07}

\end{document}